\documentclass[dvipsnames,a4paper,11pt]{amsart}
\pdfoutput=1
\usepackage[utf8]{inputenc}
\usepackage[T1]{fontenc}
\usepackage[a4paper, total={18cm, 24cm}]{geometry}
\usepackage[foot]{amsaddr}
\usepackage{tikz}
\usepackage{appendix}
\usepackage{cleveref}
\usepackage[square,comma,numbers,sort&compress,nonamebreak]{natbib}
\usepackage{xcolor}
\usepackage{mathtools}
\usepackage{caption} 

\newtheorem{thm}{Theorem}[section]
\newtheorem{lem}[thm]{Lemma}
\newtheorem{obs}[thm]{Observation}

\newtheorem{cor}[thm]{Corollary}
\newtheorem{prob}[thm]{Problem}
\newtheorem{dfn}[thm]{Definition}

\newtheorem{conj}[thm]{Conjecture}

\newtheorem{cons}[thm]{Construction}

\newcommand{\floor}[1]{\left\lfloor{#1}\right\rfloor}
\newcommand{\ceil}[1]{\left\lceil{#1}\right\rceil}

\begin{document}

\title{Minimum $k$-critical-bipartite graphs:\\
the irregular Case}
\author{Sylwia Cichacz$^1$}\thanks{This work was partially supported by the Faculty of Applied Mathematics AGH UST statutory tasks within subsidy of Ministry of Science and Higher Education.}
\author{Agnieszka Görlich$^1$}
\author{Karol Suchan$^{2,1}$}
\address{$^1$ AGH University, al. A. Mickiewicza 30, 30-059 Krakow, Poland}
\address{$^2$Universidad Diego Portales, Av. Ejército Libertador 441, 8370191 Santiago, Chile}
\email{cichacz@agh.edu.pl, forys@agh.edu.pl, karol.suchan@mail.udp.cl}

\begin{abstract}
We study the problem of finding a minimum $k$-critical-bipartite graph of order $(n,m)$: a bipartite graph $G=(U,V;E)$, with $|U|=n$, $|V|=m$, and $n>m>1$, which is $k$-critical-bipartite, and the tuple $(|E|, \Delta_U, \Delta_V)$, where $\Delta_U$ and $\Delta_V$ denote the maximum degree in $U$ and $V$, respectively, is lexicographically minimum over all such graphs. $G$ is $k$-critical-bipartite if deleting at most $k=n-m$ vertices from $U$ yields $G'$ that has a complete matching, i.e., a matching of size $m$. Cichacz and Suchan~\cite{CS} solved the problem for biregular bipartite graphs. Here, we extend their results to bipartite graphs that are not biregular. We also prove tight lower bounds on the connectivity of $k$-critical-bipartite graphs.\end{abstract}


\date{\today}
\maketitle

\section{Introduction}\label{sec:introduction}

An important body of knowledge has been developed on networks prone to faults. When representing the network as a simple undirected graph $G=(V, E)$\footnote{For standard terms and notations in graph theory, the reader is referred to the textbook by Diestel~\cite{D2017}}, the faults can be modeled as vertex or edge deletions, depending if the faults occur to the nodes or links of the network, respectively. In this work, we focus on node faults.

Many applications in diverse fields consider the robustness of assignments. For example, consider a network of $m$ sensing nodes and $n$ relay nodes with $n \geq m$. The sensing nodes need to transmit their readings through the relay nodes, based on a one-to-one assignment (due to some technological considerations), using a pre-established infrastructure of links. It is natural to ask whether we can design a network such that all the sensing nodes can do their work while no more than $k=n-m$ relay nodes are faulty. This kind of network is called a $k$-critical-bipartite graph and is part of the wider field of research related to fault-tolerant networks.

Besides applications in the design of fault-tolerant networks, $k$-critical-bipartite graphs could find applications in the design of supercomputer architectures \cite{HPCReview}, flexible processes \cite{FlexProcRreview}, personnel rostering \cite{RosteringReview}, and other areas of operations research \cite{InterdictionSurvey}. Section 2 of \cite{CS} presents a brief non-exhaustive overview of connections to other areas of research.

\subsection{Fault-tolerant graphs}
Given a graph $H$ and a positive integer $k$, a graph $G$ is called \emph{$k$-fault-tolerant with respect to $H$}, denoted by  $k$-FT$(H)$, if $G-S$ contains a subgraph isomorphic to $H$ for every $S\subset V(G)$ with $|S|\leq k$. Clearly, under this definition, it is enough to check that the property holds for $S\subset V(G)$ with $|S| = k$.

Fault-tolerance was introduced by Hayes~\cite{Hayes1976} in 1976 as a graph theoretic model of computer or communication networks working correctly in the presence of faults. Therein, the main motivation for the problem of constructing $k$-fault-tolerant graphs lies in finding fault-tolerant network architectures. A graph $H$ represents the desired interconnection network and a $k$-FT$(H)$ graph $G$ allows one to emulate the graph $H$ even in the presence of $k$ vertex (processor) faults. 

The problem has been systematically studied in different aspects. Clearly, given a graph $H$, a complete graph on $|V(H)|+k$ vertices is $k$-FT$(H)$. So it is interesting to study different quality measures motivated by diverse applications. Hayes~\cite{Hayes1976} and Ajtai et al.~\cite{AABC} considered $k$-FT$(H)$ graphs with $|V(H)|+k$ vertices and the number of edges as small as possible. A different quality measure of $k$-FT$(H)$ graphs was introduced by Ueno et al. \cite{UBHS1993}, and independently by Dudek et al.~\cite{DSZ}, where the authors were interested in $k$-FT$(H)$ graphs having as few edges as possible, disregarding the number of vertices (see also~\cite{Z2}, ~\cite{PZ}). Yet another setup was studied by Alon and Chung~\cite{AC}, Ueno and Yamada~\cite{UY}, and Zhang~\cite{Zhang}. They allowed $O(t)$ spare vertices in $k$-FT$(H)$ graphs and focused on minimizing the maximum degree (giving priority to the scalability of a network). Other results on $k$-fault-tolerance can be found, for example in \cite{ZAK2014421}.

\subsection{$k$-critical-bipartite graphs}

It is well known that the bipartite graphs are exactly the graphs that are $2$-colorable. Throughout the paper, we will use the notation $G=(U, V; E)$ for a bipartite graph $G$ with color classes $U$ and $V$. Let $|U|=n$ and $|V|=m$. We say that $G$ is of order $(n, m)$. We say that $G$ is \textit{biregular} if the degrees of the vertices in both color classes are constant, and {\em irregular} otherwise. Let $\delta_U(G)$, $\Delta_U(G)$, $\delta_V(G)$, $\Delta_V(G)$, denote the minimum and maximum degree in $G$ of a vertex in $U$ and $V$, respectively. Where it does not lead to confusion, we do not mention the graph explicitly, for example, stating just $\delta_U$ instead of $\delta_U(G)$. If $\delta_U = \Delta_U = a$ and $\delta_V = \Delta_V = b$, then we say that $G$ is $(a,b)$-regular. A complete graph of order $n$ is denoted $K_n$ and a complete bipartite graph of order $(n,m)$ is denoted $K_{n,m}$.

A $k$-critical-bipartite graph $G=(U, V;E)$, with $|U|=n$ and $|V|=m$, such that $k=n-m\geq 0$ can be seen as a $k$-FT$(H)$ graph where $H$ is a matching of size $|V|$ and the $k$ faults can occur only in $U$. Cichacz and Suchan~\cite{CS} introduced the problem of finding a minimum $k$-critical-bipartite graph according to the following definition.

\begin{dfn}[\cite{CS}]\label{MkCBG}
A bipartite graph $G=(U,V;E)$, with $|U|=n$, $|V|=m$, and $n>m>1$, is a Minimum $k$-Critical-Bipartite Graph of order $(n,m)$, M$k$CBG-$(n,m)$, if it is $k$-critical-bipartite, and the tuple $(|E|, \Delta_U, \Delta_V)$ is lexicographically minimum over all such graphs. 
\end{dfn}

Note that, given integer $n$ and $m$ with $n>m>1$, the graph $G^*(U,V;E)$ with $|U|=n$ and $|V|=m$, obtained by taking a matching of size $m$ and adding to $U$ another $k=n-m$ vertices adjacent to every vertex in $V$ gives a graph that is minimal $k$-critical-bipartite, i.e., removing any edge from $G^*$ yields a graph that is not $k$-critical-bipartite, but is not minimum according to the definition given above. Indeed, $\Delta_U(G^*) = m$, whereas we show in this paper that there exist $k$-critical-bipartite graphs $G=(U,V;E)$ that also have $|E(G)|=m(n-m+1)$, but with $\Delta_U(G) = \ceil{\frac{m(n-m+1)}{n}}$. So a minimum $k$-critical-bipartite graph of order $(n,m)$ can not be obtained by simply taking any $k$-critical-bipartite graph and removing edges, one by one, as long as the property is preserved.

Cichacz and Suchan~\cite{CS} solved the problem of finding M$k$CBG-$(n,m)$ in the case of biregular graphs, leaving open the case of irregular bipartite graphs. We solve it in this paper.

\subsection{Related work}\label{subsec:rw}

The concept of a $k$-critical-bipartite graph stems from older studies related to matchings.

In a graph $G$ of even order $n$, a \textit{perfect matching} (or 1-\textit{factor}) $M$, is a matching containing $n/2$ edges. In other words, a perfect matching covers every vertex of $G$. 

Let $G$ be a graph of order $n$ with a perfect matching $M$, and let $k$, $n/2 > k \geq 0$, be an integer. A graph $G$ of even order $n\ge 2k+2$ is called \textit{$k$-extendable} if every matching of size $k$ in $G$ extends to (i.e., is a subset of) a perfect matching in $G$. This concept was introduced by Plummer in 1980~\cite{Plummer}. 

By the following result of Plummer, $k$-extendability of a bipartite graph of order $2(n+k)$ can be seen as fault-tolerance for $H$ being a matching of size $n$, under attacks that consist in removing (at most) $k$ vertices from each color class. 

\begin{thm}[\cite{Plummer2}]\label{thm:extendability_ft}
    Let $G$ be a connected bipartite graph on $n$ vertices with the color classes $(U, V)$. Suppose $k$ is a positive integer such that $k \leq (n-2)/2$. Then the following are equivalent:
    \begin{enumerate}
        \item $G$ is $k$-extendable,
        \item $|U|=|V|$ and for each non-empty subset $X$ of $U$ such that $|X| \leq |U|-k$, there is $|N(X)|\geq |X|+k$.
        \item For all $U' \subset U$ and $V' \subset V$, $|U'|=|V'|=k$, the graph $G'=G - U' - V'$ has a perfect matching.
    \end{enumerate}
\end{thm}

There is a close relation between $k$-extendability and $k$-factor-criticality. A graph $G$ is called \textit{$k$-factor-critical} (also called simply \textit{$k$-critical}) if, after deleting any $k$ vertices, the remaining subgraph has a perfect matching. This concept was first introduced and studied for $k = 2$ by Lov\'asz \cite{Lovasz} under the name of a \textit{bicritical graph}. For $k>2$ it was introduced by Yu in 1993 \cite{Yu} and independently by Favaron in 1996 \cite{Favaron}.

It is straightforward that a bipartite graph cannot be $k$-critical. Li and Nie amended the definition of a $k$-critical graph with respect to bipartite graphs \cite{LN}. It requires that the $k$ vertices to be deleted lie in the color class with more vertices. Formally, a bipartite graph $G=(U, V; E)$ such that $k=|U|-|V|\geq 0$ is a \textit{$k$-critical-bipartite graph} if, after deleting any $k$ vertices {from the set $U$}, the remaining subgraph has a perfect matching - and this is the definition that we are using.

The problem of designing $k$-critical graphs (for the class of general graphs) with the minimum number of edges was studied by Zhang et al. in \cite{zhang2012minimum}. Using the notation $k$-FT($pK_c$), with positive integers $k$, $p$, and $c$, for a graph in which the removal of $k$ vertices leaves a subgraph that contains $p$ disjoint copies of $K_c$, the authors gave a construction for $k$-FT($pK_2$) graphs of minimum size for any generally feasible values of $p$ and $k$. This result was extended to higher values of $c$ by Cichacz et al~\cite{kliki}, who characterized minimum $k$-FT($pK_c$) graphs for $k=1$, any positive integer $p$, and $c>3$. Zhang et al. in \cite{zhang2012minimum} also gave a construction for minimum size $k$-extendable bipartite graphs. 

\subsection{Structure of the paper}

The structure of this paper is as follows. In Section \ref{sec:main_problem}, we detail the problem of designing a minimum $k$-critical-bipartite graph. In Section \ref{sec:pos_construction}, we give a construction that yields a minimum $k$-critical-bipartite graph of order $(n,m)$ for any values of $n$ and $m$ such that $n>m>1$, with $k=n-m$. We show that a $k$-critical-bipartite graph $G=(U,V;E)$ of order $(n,m)$ is minimum if $|E|=m(n-m+1)$, $\Delta(U)=\ceil{\frac{m(n-m+1)}{n}}$, and $\Delta(V)=n-m+1$. In Section \ref{sec:neg_construction}, we give a construction that yields graphs $G=(U,V;E)$ of order $(n,m)$ that also have $|E|=m(n-m+1)$, $\Delta(U)=\ceil{\frac{m(n-m+1)}{n}}$, and $\Delta(V)=n-m+1$, but are not $k$-critical-bipartite - so these properties are not sufficient for a graph to be $k$-critical-bipartite. In Section \ref{sec:connectivity} we present tight lower bounds for the connectivity of $k$-critical-bipartite graphs. And we conclude with some final remarks in Section \ref{sec:final_remarks}.

\section{Main problem}\label{sec:main_problem}

Let $G=(U, V;E)$ be a bipartite graph, with $|U|=n$ and $|V|=m$, such that $k=n-m > 0$. Let $\tilde{G}=(U, V \cup D; E \cup E^D)$ be the graph obtained from $G$ by adding to $V$ a set $D$ of $k$ vertices and making them adjacent to all vertices in $U$. Li and Nie \cite{LN} gave the following characterization of $k$-critical-bipartite graphs.

\begin{thm}[\cite{LN}]\label{th:tilde}
$G$ is $k$-critical-bipartite if and only if $\tilde{G}$ is $k$-extendable.
\end{thm}

They also described the connectivity of $k$-critical-bipartite graphs in the following theorem.
\begin{thm}[\cite{LN}] Let $G=(U, V;E)$ be a bipartite graph such that $k=|U|-|V|> 0$.  If $G$ is $k$-critical-bipartite, then $G$ is connected.\label{LN}
\end{thm}

On the other hand, Laroche et al. \cite{LMMR} gave a Hall-style characterization of $k$-critical-bipartite graphs as follows:

\begin{thm}[\cite{LMMR}]\label{LMMR} Let $G=(U, V;E)$ be a bipartite graph such that $k=|U|-|V|> 0$. The graph $G$ is $k$-critical-bipartite if and only if $|N(V')|\geq |V'|+k$ for all $\emptyset \neq V'\subseteq V$.
\end{thm}

Note that a $k$-critical-bipartite graph needs to have at least $(k+1)m$ edges. Indeed, suppose that the total number of edges is smaller. Then at least one vertex $v$ in $V$ is connected to at most $k$ distinct vertices in $U$. And there is a fault scenario where precisely the vertices in the neighborhood of $v$ are removed, in which case $v$ cannot be matched. A contradiction.

As Zhang et al.~\cite{ZZ12} for $k$-extendable bipartite graphs and Cichacz and Suchan~\cite{CS} for $k$-critical-bipartite biregular graphs, we want to study topologies where not only the total number of links is low, but also the maximum number of links per node is small (in both color classes). Thus, for given positive integer values $n,m$ such that $n>m>1$ and $k=n-m$, we want to find a bipartite graph $G=(U, V; E)$ of order $(n,m)$ that is a $k$-critical-bipartite graph and is lexicographically minimum with respect to $(|E|, \Delta_U, \Delta_V)$ (see Definition \ref{MkCBG}).

The construction below is a generalization of the construction from \cite{CS}. Indeed, the construction was used only for integers $m,n$ such that $n>m>1$ and $a = \frac{(k+1)m}{n}$ is an integer. In the construction and throughout the paper we use the following notation: $[o]=\{0,1,\ldots,o-1\}$ for any positive integer $o$.

\begin{cons}\label{cons1}
Let $n, m, a$ be positive integers such that $n>m>1$. Let $\widehat{G}_{n,m}^a=(U,V; E)$ be a bipartite graph with $U=\{u_i \mid i \in [n]\}$, $V=\{v_j \mid j \in [m]\}$, and
$$E = \left\{ (u_i, v_{(j+\alpha) \bmod{m}}) \mid i \in [n], \alpha \in [a], j=\ceil{\frac{i m}{n}}\right\}.$$
\end{cons}

Cichacz and Suchan~\cite{CS} proved that, if $n, m$ and $a$ are positive integers such that $n>m>1$, $k=n-m$, and $a n = (k+1) m$, then the graph $\widehat{G}_{n,m}^a=(U,V; E)$ is a $(a, k+1)$-regular $k$-critical-bipartite graph of size $(k+1) m$ that is M$k$CBG-$(n,m)$. Moreover, they stated the following conjecture for irregular $k$-critical-bipartite graphs.

\begin{conj}[\cite{CS}]
Let $n,m$ be positive integers such that $n>m>1$,  Let $k=n-m$ and $a=\frac{m (k+1)}{n}$ is not an integer. Then $\widehat{G}_{n,m}^{\ceil{a}}$ obtained by Construction \ref{cons1} is $k$-critical-bipartite.\label{conjCS}
\end{conj}

In this paper, we prove that the conjecture is true. Moreover, for any pair $n, m$ of positive integers such that $n>m>1$, $k=n-m$, we construct a bipartite graph $G=(U, V; E)$ of order $(n,m)$ that is $k$-critical-bipartite and is lexicographically minimum among all such graphs with respect to $(|E|, \Delta_U, \Delta_V)$. In other words, we solve the problem of finding a Minimum $k$-Critical Bipartite Graph of order $(n,m)$ (M$k$CBG-$(n,m)$) completely.

\section{Positive Construction}\label{sec:pos_construction}
In this section, we give a construction that yields Minimum $k$-Critical Bipartite Graphs of order $(n,m)$. 

Let us start by recalling the following lemma that was proved in \cite{CS}.

\begin{lem}[Lemma 3.4 in \cite{CS}]\label{number_of_i_1}
Let $x$, $y$, $c$ be positive integers such that $x>y$. Let $n=cx$, $m=cy$, and $j \in [m]$. Then the number of integer solutions to $\ceil{\frac{im}{n}} \equiv j \pmod{m}$ with respect to $i$, with $i \in [n]$, is equal to $\floor{j\frac{x}{y}} - \floor{(j-1)\frac{x}{y}}$. Moreover:
\begin{itemize}
	\item $\floor{j\frac{x}{y}} - \floor{(j-1)\frac{x}{y}} = \floor{r\frac{x}{y}} - \floor{(r-1)\frac{x}{y}}$,
where $r = j \bmod y$.

\item For any interval of consecutive $y$ values of $j$, for $(x \bmod y)$ of them, there are $\ceil{x/y}$ solutions and, for the remaining $(y - x \bmod y)$, there are $\floor{x/y}$ solutions.

\item In general, the number of solutions is $\ceil{\frac{x}{y}}$ for $(n \bmod m)$, and $\floor{\frac{x}{y}}$ for $(m - n \bmod m)$ values of $j\in[m]$.
\end{itemize}
\end{lem}

With Lemma \ref{number_of_i_1}, we can prove the following lemma that allows us to further analyze the neighborhoods of vertices in graphs like the ones from Construction \ref{cons1}. 

\begin{lem}\label{lem:remove_max}
Let $n, m$ be positive integers such that $n>m>1$. Let $j \in [m]$. Then
$$\max\left\{i \mid i \in [n], \left\lceil  \frac{im}{n}\right\rceil \equiv j\pmod{m}\right\} = \floor{\frac{j n}{m}}.$$

\end{lem}
\begin{proof}
Let us write $i_j=\max\left\{i \mid i \in [n], \left\lceil  \frac{im}{n}\right\rceil \equiv j \pmod{m}\right\}$.

We have $i_0=0=\floor{\frac{0 \cdot n}{m}}$. For $j>0$, it is easy to check that $i_j$ is equal to $i_{j-1}$ plus the number of integer solutions to $\ceil{\frac{im}{n}} \equiv j \pmod{m}$. By Lemma~\ref{number_of_i_1}, the number of integer solutions to $\ceil{\frac{im}{n}} \equiv j \pmod{m}$, with respect to $i$, with $i \in [n]$, is equal to $\floor{j\frac{n}{m}} - \floor{(j-1)\frac{n}{m}}$. Therefore, we have $i_j=i_{j-1}+\floor{j\frac{n}{m}} - \floor{(j-1)\frac{n}{m}}$. By the telescopic property, we have $i_j=\sum_{i=1}^{j}\left(\floor{i\frac{n}{m}} - \floor{(i-1)\frac{n}{m}}\right)=\floor{j\frac{n}{m}}$ for $j>0$. 
\end{proof}

It is easy to check that the edge set of the graph from Construction \ref{cons1}, in the case where $a n = (k+1) m$, can also be written as:

$$E = \left\{ (u_{(i-\beta) \bmod{n}}, v_{j}) \mid j \in [m], \beta \in [k+1], i=\max\left\{i \mid i \in [n], \left\lceil  \frac{im}{n}\right\rceil\equiv j\pmod{m}\right\} \right\}.$$

So, by Lemma~\ref{lem:remove_max}, when $a = \frac{m(k+1)}{n}$ is an integer, the edge set of the graph from Construction \ref{cons1} can be described in the following way:

$$E = \left\{ (u_{(i-\beta) \bmod{n}}, v_{j}) \mid j \in [m], \beta \in [k+1], i= \floor{\frac{j n}{m}}\right\}.$$

Let us generalize this construction to the case where $a = \frac{m(k+1)}{n}$ is not an integer.

\begin{cons}\label{cons2}
Let $n, m$ be positive integers such that $n>m>1$. Let $k=n-m$. Let $\overline{G}_{n,m}=(U,V; E)$ be a bipartite graph with $U=\{u_i \mid i \in [n]\}$, $V=\{v_j \mid j \in [m]\}$, and
$$E = \left\{ (u_{(i-\beta) \bmod{n}}, v_{j}) \mid j \in [m], \beta \in [k+1], i= \floor{\frac{j n}{m}}\right\}.$$
\end{cons}

\begin{figure}[ht!]
\begin{center}
\begin{tikzpicture}[scale=1,style=thick,x=1cm,y=1cm]
\def\vr{3pt} 

\path (0,1) coordinate (u6);
\path (0,2) coordinate (u5);
\path (0,3) coordinate (u4);
\path (0,4) coordinate (u3);
\path (0,5) coordinate (u2);
\path (0,6) coordinate (u1);
\path (2,2) coordinate (v5);
\path (2,3) coordinate (v4);
\path (2,4) coordinate (v3);
\path (2,5) coordinate (v2);
\path (2,6) coordinate (v1);

\draw (u6) -- (v1) -- (u1) -- (v2) -- (u2) -- (v3) -- (u3) -- (v4) -- (u4) -- (v5) -- (u5);

\draw (u1) [fill=white] circle (\vr);
\draw (u2) [fill=white] circle (\vr);
\draw (u3) [fill=white] circle (\vr);
\draw (u4) [fill=white] circle (\vr);
\draw (u5) [fill=white] circle (\vr);
\draw (u6) [fill=white] circle (\vr);
\draw (v1) [fill=white] circle (\vr);
\draw (v2) [fill=white] circle (\vr);
\draw (v3) [fill=white] circle (\vr);
\draw (v4) [fill=white] circle (\vr);
\draw (v5) [fill=white] circle (\vr);

\draw[anchor = east] (u6) node {$u_5$};
\draw[anchor = east] (u5) node {$u_4$};
\draw[anchor = east] (u4) node {$u_3$};
\draw[anchor = east] (u3) node {$u_2$};
\draw[anchor = east] (u2) node {$u_1$};
\draw[anchor = east] (u1) node {$u_0$};
\draw[anchor = west] (v5) node {$v_4$};
\draw[anchor = west] (v4) node {$v_3$};
\draw[anchor = west] (v3) node {$v_2$};
\draw[anchor = west] (v2) node {$v_1$};
\draw[anchor = west] (v1) node {$v_0$};

\path (5,1) coordinate (1u6);
\path (5,2) coordinate (1u5);
\path (5,3) coordinate (1u4);
\path (5,4) coordinate (1u3);
\path (5,5) coordinate (1u2);
\path (5,6) coordinate (1u1);
\path (7,2) coordinate (1v5);
\path (7,3) coordinate (1v4);
\path (7,4) coordinate (1v3);
\path (7,5) coordinate (1v2);
\path (7,6) coordinate (1v1);

\draw (1u6) -- (1v1) -- (1u1) -- (1v2) -- (1u2) -- (1v3) -- (1u3) -- (1v4) -- (1u4) -- (1v5) -- (1u5);
\draw (1u5) -- (1v1);
\draw (1u6) -- (1v2);

\draw (1u1) [fill=white] circle (\vr);
\draw (1u2) [fill=white] circle (\vr);
\draw (1u3) [fill=white] circle (\vr);
\draw (1u4) [fill=white] circle (\vr);
\draw (1u5) [fill=white] circle (\vr);
\draw (1u6) [fill=white] circle (\vr);
\draw (1v1) [fill=white] circle (\vr);
\draw (1v2) [fill=white] circle (\vr);
\draw (1v3) [fill=white] circle (\vr);
\draw (1v4) [fill=white] circle (\vr);
\draw (1v5) [fill=white] circle (\vr);

\draw[anchor = east] (1u6) node {$u_5$};
\draw[anchor = east] (1u5) node {$u_4$};
\draw[anchor = east] (1u4) node {$u_3$};
\draw[anchor = east] (1u3) node {$u_2$};
\draw[anchor = east] (1u2) node {$u_1$};
\draw[anchor = east] (1u1) node {$u_0$};
\draw[anchor = west] (1v5) node {$v_4$};
\draw[anchor = west] (1v4) node {$v_3$};
\draw[anchor = west] (1v3) node {$v_2$};
\draw[anchor = west] (1v2) node {$v_1$};
\draw[anchor = west] (1v1) node {$v_0$};


\draw(1,0.5) node {$\overline{G}_{6,5}$};
\draw(6,0.5) node {$\widehat{G}_{6,5}^{2}$};

\end{tikzpicture}
\caption{The graphs $\overline{G}_{6,5}$ and $\widehat{G}_{6,5}^{2}$}\label{newcons}
\end{center}
\end{figure}
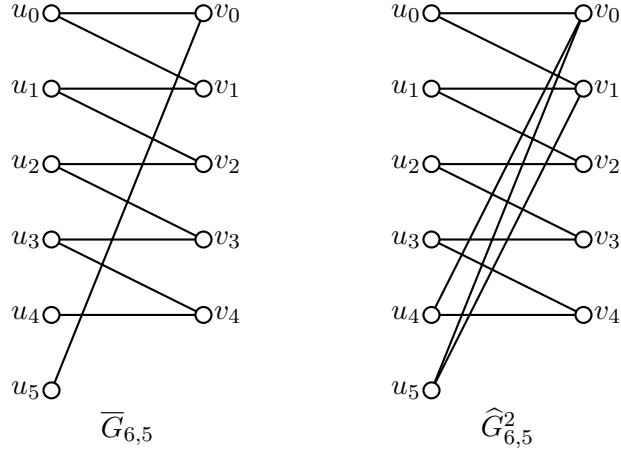

So the following observation holds.

\begin{obs}\label{obs_cons_eq}
If $a = \frac{m(k+1)}{n}$ is an integer then  $\overline{G}_{n,m}=\widehat{G}_{n,m}^{a}$.
\end{obs}

See Figure~\ref{newcons} for a small example of graphs $\widehat{G}_{n,m}^{\ceil{a}}$ and $\overline{G}_{n,m}$ in which $a = \frac{m(k+1)}{n}$ is not an integer. Let us present a lemma that relates the graphs obtained by Constructions \ref{cons1} and \ref{cons2} in general.

\begin{lem}\label{subgraph}
Let $n,m$ be positive integers such that $n>m>1$. Let $k=n-m$ and $a=\frac{m(k+1)}{n}$. Then the graph $\overline{G}_{n,m}=(U,V; \overline{E})$ obtained by Construction \ref{cons2} is a subgraph of the graph $\widehat{G}_{n,m}^{\ceil a}=(U,V; \widehat E)$ given in Construction \ref{cons1}.
\end{lem}
\begin{proof}
Let us consider the graph $\overline{G}_{n,m}=(U,V; \overline{E})$ obtained by Construction \ref{cons2}. Let $e_0 \in \overline E$. Then
there is $j_0 \in [m]$ and $\beta_0 \in [k+1]$ such that 
$e_0=(u_{(i_0 - \beta_0) \bmod n},v_{j_0})$ for $i_0=\floor{\frac{j_0n}{m}}$. Let us show that $e_0$ also belongs to $\widehat E$.

By Construction \ref{cons1}, there is $\{(u_{(i_0 - \beta_0)},v_{(j_{(i_0-\beta_0)}+\alpha) \bmod{m} }), \alpha \in [\ceil a]\} \subset \widehat E$, where $j_{(i_0-\beta_0)}=\ceil{\frac{(i_0-\beta_0)m}{n}}$. Let us choose $\gamma \in [m]$ such that 
$$\frac{(i_0-\beta_0)m}{n}=\frac{(\floor{\frac{j_0n}{m}}-\beta_0)m}{n}=\frac{(\frac{j_0n-\gamma}{m}-\beta_0)m}{n}=j_0-\frac{\gamma +\beta_0 m}{n}.$$ 
So, there is $j_{(i_0-\beta_0)}=\ceil{j_0-\frac{\gamma +\beta_0 m}{n}}=j_0-\floor{\frac{\gamma +\beta_0 m}{n}}$. 

Note that 
$$\floor{\frac{\gamma +\beta_0 m}{n}} \leq \floor{\frac{m-1+km}{n}}=\floor{a-\frac{1}{n}}=\begin{cases} a-1 \; {\rm if} \; a \in \mathbb N \\  b, \; b \in \{\ceil a-2, \ceil a-1\} \; \rm{otherwise.} \end{cases}$$
So there exists $\alpha_0 \in [\ceil a]$ such that $j_0=j_{(i_0-\beta_0)}+\alpha_0$, and hence $$e_0=(u_{(i_0-\beta_0)},v_{j_0}) \in \{(u_{(i_0-\beta_0)},v_{(j_{(i_0-\beta_0)}+\alpha) \bmod m}), \alpha \in [ \ceil a]\} \subset \widehat E.$$

\end{proof}

\begin{thm}\label{positive} 
Let $n,m$ be positive integers such that $n>m>1$. Let $k=n-m$. Then $\overline{G}_{n,m}=(U,V; E)$ obtained by Construction \ref{cons2} is a minimum $k$-critical-bipartite graph of order $(n,m)$.
\end{thm}
\begin{proof}
Let us show that, for any $S\neq\emptyset$, $S\subset V$, there is $|N_{\overline{G}_{n,m}}(S)|\geq |S|+k$, which, by Theorem~\ref{LMMR}, implies that $\overline{G}_{n,m}$ is $k$-critical. The proof is by induction on $|S|$. 

Let $|S|=1$. Let $v_j$, $j \in [m]$ be the vertex in $S$. By definition, $N_{\overline{G}_{n,m}}(v_j)=\{(u_{(i-\beta) \bmod{n}}, v_{j}) \mid i= \floor{\frac{j n}{m}}, \beta \in [k+1]\}$, so the conclusion is true. 

Given an integer $p$, $1 \leq p \leq m-1$, suppose that $|N_{\overline{G}_{n,m}}(S')|\geq p+k$ holds for any $S'$, $S' \subset V$, such that $|S'| = p$. 

Take any $S$, $S\subset V$, with $|S| = p + 1$. 
Suppose first that there exists $v \in S$ such that $N_{\overline{G}_{n,m}}(v) \setminus N_{\overline{G}_{n,m}}(S \setminus v) \neq \emptyset$ and hence $|N_{\overline{G}_{n,m}}(v) \cap N_{\overline{G}_{n,m}}(S \setminus \{v\})|< \deg_{\overline{G}_{n,m}}(v)$. Let $S'= S \setminus\{v\}$. Then $|S'| = p$ and, by the induction hypothesis, $|N_{\overline{G}_{n,m}}(S')|\geq p+k$. Hence, $$|N_{\overline{G}_{n,m}}(S)|\geq p+k +|N_{\overline{G}_{n,m}}(v)\setminus N_{\overline{G}_{n,m}}(S')|\geq p+k+1=|S|+k.$$

Assume now that $N_{\overline{G}_{n,m}}(v) \subset N_{\overline{G}_{n,m}}(S \setminus v)$ for every $v \in S$. Let us show that it implies that $N_{\overline{G}_{n,m}}(S) = U$, and so $|N_{\overline{G}_{n,m}}(S)| = |V|+k \geq |S|+k$. 

Suppose, to the contrary, that $I = \{i \in [n] \colon u_i \not\in N_{\overline{G}_{n,m}}(S)\} \neq \emptyset$. Let $i_0 = \max\{i \colon i \in I, (i+1) \bmod n \notin I\}$. Then there exists $v_r \in S$ such that $u_{(i_0+1) \bmod n} \in N_{\overline G_{n,m}}(v_r)$. Since $N_{\overline{G}_{n,m}}(v_r)\subset N_{\overline{G}_{n,m}}(S\setminus \{v_r\}),$ there exists  $v_l \in S$, $r\neq l$, such that  $u_{(i_0+1) \bmod n}\in  N_{\overline{G}_{n,m}}(v_l)$. By Construction \ref{cons2}, $N_{\overline{G}_{n,m}}(v_{j})=\{u_{i_{j}-k},u_{i_{j}-(k-1)},\ldots,u_{i_{j}}\}$ for every $j\in[m]$. It is easy to check that, for any $j_1,j_2 \in [m]$, $j_1\neq j_2$ implies $i_{j_1}\neq i_{j_2}$. So there is $\begin{vmatrix}N_{\overline{G}_{n,m}}(v_{j_1})\setminus N_{\overline{G}_{n,m}}(v_{j_2})\end{vmatrix}\geq 1$. 

On the other hand, since $u_{(i_0+1) \bmod n}\in  N_{\overline{G}_{n,m}}(v_l) \cap  N_{\overline{G}_{n,m}}(v_r)$ and $u_{i_0 \bmod n}\notin  N_{\overline{G}_{n,m}}(v_l) \cup  N_{\overline{G}_{n,m}}(v_r)$, there is $N_{\overline{G}_{n,m}}(v_l)=N_{\overline{G}_{n,m}}(v_r)=\{u_{i_0+1},u_{i_0+2},\ldots,u_{i_0+k+1}\}$, a contradiction.

It is easy to check, by the pigeonhole principle, that there is $\delta_V\geq k+1$ for any $k$-critical-bipartite graph $G=(U, V; E)$ of order $(n,m)$. By definition, $\overline{G}_{n,m}$ has $(k+1)m$ edges. So the construction is minimum.
\end{proof}

Combining Theorem \ref{positive} and Lemma \ref{subgraph}, we obtain that the Conjecture~\ref{conjCS} is true.

\begin{cor}
Let $n,m$ be positive integers such that $n>m>1$,  Let $k=n-m$ and $a=\frac{m (k+1)}{n}$ is not an integer. Then $\widehat{G}_{n,m}^{\ceil{a}}$ obtained by Construction \ref{cons1} is $k$-critical-bipartite.\label{corCS}
\end{cor}

\section{Negative Construction}\label{sec:neg_construction}

In this section we give a construction that yields graphs $G=(U,V;E)$ of order $(n,m)$ that also have $|E|=m(k+1)$, $\Delta(U)=\ceil{\frac{m(k+1)}{n}}$, and $\Delta(V)=k+1$, where $k=n-m$, but are not $k$-critical-bipartite. So these properties are not sufficient for a graph to be $k$-critical-bipartite. 

Note that Theorem \ref{positive}, together with the simple observation that there is $\delta(V) \geq k+1$ in any $k$-critical-bipartite $G=(U,V;E)$ with $n>m>1$ and $k=n-m$, implies that there is $\delta(V) = \Delta(V) = k+1$, $\Delta(U) =\ceil{\frac{m(k+1)}{n}}$, {and $\delta(U) \leq k$} if $G=(U,V;E)$ is minimum $k$-critical-bipartite. Both in Construction \ref{cons2} and Construction \ref{cons4} that follows, there is $\delta(U) =\floor{\frac{m(n-m+1)}{n}}$. The graphs obtained by the two constructions have the same degree sequences, so even fixing the vertex degrees does not make a graph $k$-critical-bipartite.

Cichacz and Suchan in \cite{CS} gave the following construction of a class of biregular graphs.
\begin{cons}[\cite{CS}]\label{ncons1}
Let $n, m$ be positive integers such that $n>m>1$,  and $a=\frac{(n-m+1)m}{n}$ is an integer. Let $\gcd(n,m)=c$, $n=cx$, and $m=cy$. Let $\check{G}_{n,m}^a=(U,V; E)$ be the bipartite graph with $U=\{u_i \mid i \in [n]\}$, $V=\{v_j \mid j \in [m]\}$, $E = \{ (u_i, v_{(j+\alpha) \bmod{m}}) \mid i \in [n], \alpha \in [a], j=\floor{\frac{i}{x}}y\}$.
\end{cons}

It is easy to check that the graph $\check{G}_{n,m}^a=(U, V; E)$ can also be constructed in the following way. Let $b=n-m+1$ and $d=\gcd(a,b)$. Note that there is $a=dy$ and $b=dx$ (see \cite{CS}). Let $\check{G}$ be a $d$-regular bipartite graph having color classes $U^{\prime}=\{u_i \mid i \in [c]\}$ and $V^{\prime}=\{v_j \mid j \in [c]\}$, where $c=\gcd(n,m)$, such that $E(\check{G})=\{u_iv_{(i+\delta)\bmod{c}},\; i \in [c], \delta \in [d] \}$. We construct the graph $\check{G}_{n,m}^a=(U,V; E)$ by ``blowing up'' each vertex $u_i$ into $x=n/c$ vertices $u_{i,\alpha}$, $\alpha \in [x]$, and each vertex $v_j$ into $y=m/c>1$ vertices $v_{j,\beta}$, $\beta \in [y]$. Each edge from $\check{G}$ is substituted by the corresponding complete bipartite graph $K_{x,y}$. Note that $\check{G}_{n,m}^a$ is $(a,b)$-regular.

The authors showed that, despite having the same degrees as minimum biregular $k$-critical-bipartite graphs, the graphs obtained by Construction \ref{ncons1} tend not to be $k$-critical-bipartite.

\begin{obs}[\cite{CS}]\label{kcons1}
Let $n, m$ be positive integers such that $n>m>1$, $k=n-m$, and $a=\frac{m (k+1)}{n}$ is an integer. The graph $\check{G}_{n,m}^a$  given in Construction~\ref{ncons1} is biregular $k$-critical-bipartite if and only if $\gcd(n,m)=m$.
\end{obs}

Cichacz and Suchan~\cite{CS} complemented Construction \ref{ncons1} with another construction for the case where $\gcd(n,m)=m$ and $a = \frac{m(n-m+1)}{n}$ is an integer to get the following result.

\begin{obs}[\cite{CS}]
Let $n=|U|$, $m=|V|\in $ be such positive integers that $1 < m < n$, $k=n-m$ and $a = \frac{m(k+1)}{n}$ is an integer. There exists an $(a, k+1)$-regular bipartite graph $G = (U, V ;E)$ that is not $k$-critical if and only if $a < m - 1$.
\end{obs}

So we know constructions of $(a,b)$-regular bipartite graphs of order $(n,m)$ that are not $k$-critical-bipartite, where $b=n-m+1$ and $a=\frac{m(k+1)}{n}$, whenever such graphs exist. In what follows, we focus on the cases where $a=\frac{m(k+1)}{n}$ is not an integer.

Let us recall the results of Havel-Hakimi on constructing bipartite graphs based on degree sequences that are useful for constructing graphs that are not $k$-critical-bipartite. Let $P : p_0 \geq p_1 \geq \ldots \geq p_{n-1}$ and $Q : q_0 \geq q_1 \geq \ldots \geq q_{n-1}$ be sequences of non-negative integers. The pair $(P, Q)$ is \textit{bigraphic} if there exists a bipartite graph $G=(U,V; E)$ with $|U|=n$ and $|V|=m$ in which $P$ and $Q$ describe the degrees of the vertices in $U$ and $V$, respectively.  The following theorem is a version of Havel-Hakimi’s theorem for bigraphic sequences.

\begin{thm}[\cite{West}]\label{West}
The pair $(P, Q)$ is bigraphic if and only if the pair $(P',Q')$ is bigraphic, where $(P',Q')$ is obtained from $(P, Q)$ by deleting the largest element $p_1$ of $P$ and subtracting one from each of the $p_1$ largest elements of $Q$.\end{thm}


Let $x, y, b$ be positive integers such that $x>y>1$ and $b\leq x$. Let $r=(b \bmod x)$, $b = dx+r$, and $l=by-x\floor{\frac{yb}{x}}$. Let $p_i=\ceil{\frac{yb}{x}}$ for $i\in[l]$ and $p_i=\floor{\frac{yb}{x}}$ for $i\in\{l,\ldots,x-1\}$, let $q_j=b$ for $j\in[y]$. Let $P=(p_0, p_1, \ldots, p_{x-1})$ and $Q=(q_0, q_1, \ldots, q_{y-1})$. Note that there is $r=by-dx=by-x\floor{\frac{yb}{x}}$ and $by=(x-r)d+r(d+1)$, which implies that $\sum_{i \in [x]}p_i = by$. So the pair $(P,Q)$ is bigraphic by Theorem~\ref{West}. 
Note that  $(P,Q)$ is bigraphic if and only if $(Q,P)$ is. 

Let us present two constructions based on $P$ and $Q$ defined above: one for $(P,Q)$ and the other for $(Q,P)$. The graphs thus obtained have different properties.

\begin{figure}[ht!]
\begin{center}
\begin{tikzpicture}[scale=1,style=thick,x=1cm,y=1cm]
\def\vr{3pt} 

\path (0,1) coordinate (u6);
\path (0,2) coordinate (u5);
\path (0,3) coordinate (u4);
\path (0,4) coordinate (u3);
\path (0,5) coordinate (u2);
\path (0,6) coordinate (u1);
\path (2,2) coordinate (v5);
\path (2,3) coordinate (v4);
\path (2,4) coordinate (v3);
\path (2,5) coordinate (v2);
\path (2,6) coordinate (v1);

\path (5,1) coordinate (1u6);
\path (5,2) coordinate (1u5);
\path (5,3) coordinate (1u4);
\path (5,4) coordinate (1u3);
\path (5,5) coordinate (1u2);
\path (5,6) coordinate (1u1);
\path (7,2) coordinate (1v5);
\path (7,3) coordinate (1v4);
\path (7,4) coordinate (1v3);
\path (7,5) coordinate (1v2);
\path (7,6) coordinate (1v1);

\draw (u6) -- (v5) -- (u3) -- (v1) -- (u1) -- (v2) -- (u4) -- (v3) -- (u2) -- (v4) -- (u5);

\draw (1u1) -- (1v1) -- (1u2) -- (1v4) -- (1u1);
\draw (1u3) -- (1v2) -- (1u4) -- (1v5) -- (1u3);
\draw (1u5) -- (1v3) -- (1u6);

\draw (u1) [fill=white] circle (\vr);
\draw (u2) [fill=white] circle (\vr);
\draw (u3) [fill=white] circle (\vr);
\draw (u4) [fill=white] circle (\vr);
\draw (u5) [fill=white] circle (\vr);
\draw (u6) [fill=white] circle (\vr);
\draw (v1) [fill=white] circle (\vr);
\draw (v2) [fill=white] circle (\vr);
\draw (v3) [fill=white] circle (\vr);
\draw (v4) [fill=white] circle (\vr);
\draw (v5) [fill=white] circle (\vr);

\draw[anchor = east] (u6) node {$u_5$};
\draw[anchor = east] (u5) node {$u_4$};
\draw[anchor = east] (u4) node {$u_3$};
\draw[anchor = east] (u3) node {$u_2$};
\draw[anchor = east] (u2) node {$u_1$};
\draw[anchor = east] (u1) node {$u_0$};
\draw[anchor = west] (v5) node {$v_4$};
\draw[anchor = west] (v4) node {$v_3$};
\draw[anchor = west] (v3) node {$v_2$};
\draw[anchor = west] (v2) node {$v_1$};
\draw[anchor = west] (v1) node {$v_0$};

\draw (1u1) [fill=white] circle (\vr);
\draw (1u2) [fill=white] circle (\vr);
\draw (1u3) [fill=white] circle (\vr);
\draw (1u4) [fill=white] circle (\vr);
\draw (1u5) [fill=white] circle (\vr);
\draw (1u6) [fill=white] circle (\vr);
\draw (1v1) [fill=white] circle (\vr);
\draw (1v2) [fill=white] circle (\vr);
\draw (1v3) [fill=white] circle (\vr);
\draw (1v4) [fill=white] circle (\vr);
\draw (1v5) [fill=white] circle (\vr);

\draw[anchor = east] (1u6) node {$u_5$};
\draw[anchor = east] (1u5) node {$u_4$};
\draw[anchor = east] (1u4) node {$u_3$};
\draw[anchor = east] (1u3) node {$u_2$};
\draw[anchor = east] (1u2) node {$u_1$};
\draw[anchor = east] (1u1) node {$u_0$};
\draw[anchor = west] (1v5) node {$v_4$};
\draw[anchor = west] (1v4) node {$v_3$};
\draw[anchor = west] (1v3) node {$v_2$};
\draw[anchor = west] (1v2) node {$v_1$};
\draw[anchor = west] (1v1) node {$v_0$};

\draw(1,0.5) node {$\dot{G}_{6,5}^2$};
\draw(6,0.5) node {$\ddot{G}_{6,5}^2$};

\end{tikzpicture}
\end{center}
\caption{The graphs $\dot{G}_{6,5}^2$ and $\ddot{G}_{6,5}^2$}\label{rys}
\end{figure}
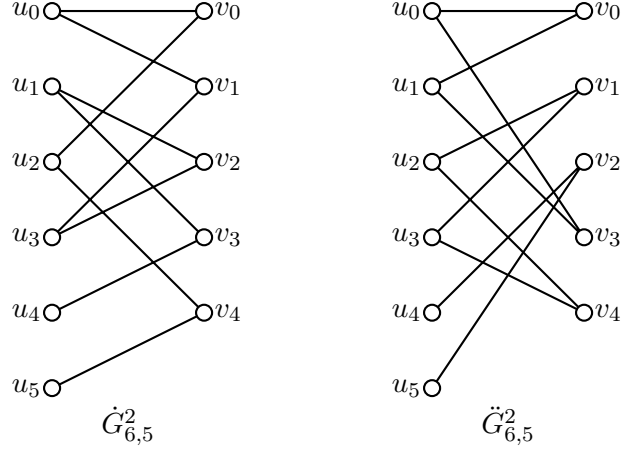

\begin{cons}\label{cons4}
Let $x, y, b$ be positive integers such that $x>y>1$ and $b\leq x$. Let $l=by-x\floor{\frac{yb}{x}}$. Let $p_i=\ceil{\frac{yb}{x}}$ for $i\in[l]$ and $p_i=\floor{\frac{yb}{x}}$ for $i\in\{l,\ldots,x-1\}$. Let $D_0=0$ and $D_i=(\sum_{j=0}^{i-1}p_i) \bmod y$ for $i \in \{1, \ldots, x-1\}$. Let $\dot{G}_{x,y}^b=(U,V; E)$ be the bipartite graph with $U=\{u_i \mid i \in [x]\}$, $V=\{v_j \mid j \in [y]\}$ such that for every $i \in [x]$:
$$N_{\dot{G}_{x,y}^b}(u_{i})=\{v_{(D_{i}+\pi)\bmod y} \mid \pi \in [p_i]\}.$$
\end{cons}

\begin{lem}\label{cons4graph}
Let $x, y, b$ be positive integers such that $x>y>1$ and $b\leq x$. Then the graph $\dot{G}_{x,y}^b=(U,V; E)$ obtained by Construction \ref{cons4} has size $|E|=by$, $\deg(u)\in\left\{\floor{\frac{yb}{x}},\ceil{\frac{yb}{x}}\right\}$ for every $u\in U$, and $\deg(v)=b$ for every $v\in V$. \end{lem}
\begin{proof}
Note that the graph $\dot{G}_{x,y}^b$ is a graph constructed as in Theorem~\ref{West} for $(P,Q)$ for $P$ and $Q$ defined above, where $P$ and $Q$ describe the degrees of vertices in $U$ and $V$, respectively. Thus $\deg(u)\in\left\{\floor{\frac{yb}{x}},\ceil{\frac{yb}{x}}\right\}$ for every $u\in U$ and $\deg(v)=b$ for every $v\in V$.
\end{proof}

\begin{cons}\label{cons5}
Let $x, y,b$ be positive integers such that $b\leq x$.
Let $\ddot{G}_{x,y}^b=(U,V; E)$ be the bipartite graph with $U=\{u_i \mid i \in [x]\}$, $V=\{v_j \mid j \in [y]\}$, such that for each $j \in [y]$:
$$N_{\ddot{G}_{x,y}^b}(v_{j})=\{u_{(jb+\beta)\bmod x} \mid \beta \in [b]\}.$$

 \end{cons}

\begin{lem}\label{cons5graph}
Let $x,y,b$ be positive integers such that  $b\leq x$. Then  $\ddot{G}_{x,y}^b=(U,V; E)$  has size $|E|=by$, $\deg(u)\in\left\{\floor{\frac{yb}{x}},\ceil{\frac{yb}{x}}\right\}$ for every $u\in U$ and $\deg(v)=b$ for every $v\in V$. \end{lem}
\begin{proof}
Note that the graph $\ddot{G}_{x,y}^b$ is a graph constructed as in Theorem~\ref{West} for $(Q,P)$ for $Q$ and $P$ defined above, where $Q$ and $P$ describe the degrees of vertices in $V$ and $U$, respectively. Thus $\deg(u)\in\left\{\floor{\frac{yb}{x}},\ceil{\frac{yb}{x}}\right\}$ for every $u\in U$ and $\deg(v)=b$ for every $v\in V$.
\end{proof}


Note that, in general, $\dot{G}_{x,y}^b\not\cong \ddot{G}_{x,y}^b$. For example, $\dot{G}_{6,5}^2$ is connected, whereas $\ddot{G}_{6,5}^2$ is not (see Figure~\ref{rys}). Based on the above constructions we are able to show the following result.


\begin{obs}\label{obs_neg}
Let $n,m$ be positive integers such that $n>m>1$. Let $k=n-m$. Let $c$ be a positive integer such that $n=cx$ and $m=cy$. If $c>1$ and $k+1\leq x$, or $d=\frac{n}{k+1}>1$ and $d$ is an integer, then there exists a graph $G=(U,V;E)$ such that $\deg(u)\in\left\{\floor{\frac{m(k+1)}{n}},\ceil{\frac{m(k+1)}{n}}\right\}$ for any $u\in U$ and $\deg(v)=k+1$ for any $v\in V$ which is not $k$-critical-bipartite.
\end{obs}
\begin{proof}
If $c>1$ and $k+1\leq x$, define a graph $\dddot{G}_{n,m}^{k+1}=(U,V; E)$ as the disjoint union of $c$ copies of $\dot{G}_{x,y}^{k+1}=(U',V'; E')$. Since $c>1$, the graph $\dddot{G}_{n,m}^{k+1}$ is disconnected, therefore is not $k$-critical-bipartite by Theorem~\ref{LN}. 

Suppose now that $d=\frac{n}{k+1}>1$ and $d$ is an integer. Then the graph $\ddot{G}^{k+1}_{n,m}$ is disconnected. Indeed, since $N_{\ddot{G}_{n,m}^{k+1}}(v_{j})=\{u_{(j(k+1)+\beta)\bmod n} \mid \beta \in [k+1]\}$, every vertex in $V$ has one of $d=\frac{n}{k+1}>1$ disjoint neighborhoods in $U$. So $\ddot{G}^{k+1}_{n,m}$ has $d$ connected components. By Theorem~\ref{LN}, it means that it is not $k$-critical-bipartite.
\end{proof}

Note that Observation \ref{obs_neg} does not cover all cases of $n$ and $m$. Assumption $n \leq k+1$ implies that we deal with a complete bipartite graph, which is $k$-critical-bipartite. So the cases that are left open have that $\frac{n}{n-m+1}$ is not an integer, and $\gcd(n,m)=1$ or $n-m+1 > \frac{n}{c}$ for any $c$ non-trivial common divisor of $n$ and $m$.

\section{Connectivity}\label{sec:connectivity}

A graph $G=(V, E)$ is said to be $k$-\textit{connected} if it has more than $k$ vertices and remains connected whenever strictly fewer than $k$ vertices are removed. The \textit{connectivity} of $G$, denoted $\kappa(G)$, is the maximum $k$ such that G is $k$-connected.

Given a graph $G$ and two vertices $u$ and $v$ that belong to the same component of $G$, a \textit{vertex cut} in $G$ separating $u$ and $v$ is a set $S$ of vertices of $G$ whose removal leaves $u$ and $v$ in different components of $G - S$. The \textit{local connectivity} $\kappa_{u, v}(G)$ of $u$ and $v$ in $G$ is the size of a smallest vertex cut separating $u$ and $v$. Given a graph $G$, $\kappa(G)$ equals the minimum $\kappa_{u, v}(G)$ over all nonadjacent pairs of vertices $u$, $v$ (except for complete graphs).

For a set $S\subset V(G)$, the \textit{set connectivity} of $S$, denoted by $\kappa_S(G)$, is the size of a smallest vertex cut separating any $u,v\in S$.

Favaron~\cite{Favaron} showed that every $k$-critical graph $G$ of order $n > k$ is $k$-connected and this result is sharp. On the other hand, Li and Nie~\cite{LN} only showed that every $k$-critical-bipartite graph $G$ is $1$-connected. We improve this result to show that for any $k$-critical-bipartite graph $G=(U, V; E)$ with $|U|=n$, $|V|=m$, and $k=n-m > 0$, there is:
\begin{enumerate}
    \item $\kappa_V(G)\geq k$,
    \item $\kappa_U(G) \geq \min\{\delta_U(G),k\}$,
    \item $\kappa(G) \geq \min\{\delta(G),k\}$.
\end{enumerate}

\begin{thm}\label{kappaV}
Let $n,m$ be positive integers such that $1<m<n$. Let $k=n-m$. Then $\kappa_V(G)\geq k$ for any $k$-critical-bipartite graph $G=(U,V;E)$ with $|U|=n$ and $|V|=m$.
\end{thm}
\begin{proof}
Let $G=(U,V;E)$ be a $k$-critical-bipartite graph of order $(n,m)$. Towards a contradiction, suppose that $\kappa_V(G)< k$. So there exists a set $Z \subset U \cup V$ with $|Z|<k$ that separates two vertices $v_1, v_2$ in $V$. Let $Z_1 = Z \cap U$ and $Z_2 = Z \cap V$. Note that there is no path between $v_1$ and $v_2$ in $G'=(U',V'; E')=G[(U\setminus Z_1)\cup(V\setminus Z_2)]$. So we can choose a partition of $G'$ into two subgraphs $G'_1=(U'_1,V'_1;E'_1)$ and $G'_2=(U'_2,V'_2;E'_2)$ that are unions of components of $G'$ with $v_1 \in V'_1$ and $v_2 \in V'_2$.
 
Let $|U'_i|= |V'_i|+\varepsilon_i$ for $i=1,2$. So there is $|V'_1|+|V'_2|+|Z_2|+k=|U'_1|+|U'_2|+|Z_1|=|V'_1|+|V'_2|+|Z_1|+\varepsilon_1+\varepsilon_2$. Thus $|Z_2|+k=|Z_1|+ \varepsilon_1+\varepsilon_2$.

By Theorem \ref{LMMR}, there is $|N_G(V'_1)| \geq |V'_1|+k$. On the other hand, since $U'_1 \cup Z_1 \supseteq N_G(V'_1)$, there is $|U'_1 \cup Z_1| \geq |N_G(V'_1)|$. So, by simplifying $|V'_1|+\varepsilon_1+|Z_1| = |U'_1| + |Z_1| \geq |V'_1|+k$, we get $\varepsilon_1+|Z_1| \geq k$. In a similar way, we get that $\varepsilon_2+|Z_1| \geq k$. 


On one hand, since $|Z_2|+k=|Z_1|+ \varepsilon_1+\varepsilon_2\geq k+\varepsilon_2$, we obtain that $|Z_2|\geq \varepsilon_2$. On the other hand, there is $|Z_1|+|Z_2|<k\leq |Z_1|+\varepsilon_2$, so $|Z_2|<\varepsilon_2$, a contradiction.
\end{proof}

\begin{thm}\label{kappaU}
Let $n,m$ be positive integers such that $1<m<n$. Let $k=n-m$. Then $\kappa_U(G) \geq \min\{\delta_U(G),k\}$ for any $k$-critical-bipartite graph $G=(U,V;E)$ with $|U|=n$ and $|V|=m$. Moreover, for every separator $Z$ in $G$ with $|Z|<k$, there exists a vertex $u \in U$ with $N(u) \subseteq Z$. 
\end{thm}

\begin{proof}
Let $G=(U,V;E)$ be a $k$-critical-bipartite graph of order $(n,m)$.  Towards a contradiction, suppose that $\kappa_U(G)< \min\{\delta_U(G),k\}$. So there exists a set $Z \subset U \cup V$ with $|Z|<\min\{\delta_U(G),k\}$ that separates two vertices $u_1, u_2$ in $U$. Let $Z_1 = Z \cap U$ and $Z_2 = Z \cap V$. Note that there is no path between $u_1$ and $u_2$ in $G'=(U',V'; E')=G[(U\setminus Z_1)\cup(V\setminus Z_2)]$. So we can choose a partition of $G'$ into two subgraphs $G'_1=(U'_1,V'_1;E'_1)$ and $G'_2=(U'_2,V'_2;E'_2)$ that are unions of components of $G'$ with $u_1 \in U'_1$ and $u_2 \in U'_2$. 

Suppose first that the graph $G'$ contains an isolated vertex $u\in U'$, then $|Z_2|\geq \delta_{U}$ since $N(u)\subset Z_2$, a contradiction. 

Assume now that the graph $G'$ does not contain an isolated vertex $u\in U'$, hence $|V_i'|>0$ for $i=1,2$. Suppose that $|Z|=|Z_1|+|Z_2|<k$. We will proceed now like in the proof of Theorem \ref{kappaV}. 

Let $|U'_i|= |V'_i|+\varepsilon_i$ for $i=1,2$. So there is $|V'_1|+|V'_2|+|Z_2|+k=|U'_1|+|U'_2|+|Z_1|=|V'_1|+|V'_2|+|Z_1|+\varepsilon_1+\varepsilon_2$. Thus $|Z_2|+k=|Z_1|+ \varepsilon_1+\varepsilon_2$.

Since $V_1'\neq \emptyset$ by Theorem \ref{LMMR}, there is $|N_G(V'_1)| \geq |V'_1|+k$. On the other hand, since $U'_1 \cup Z_1 \supseteq N_G(V'_1)$, there is $|U'_1 \cup Z_1| \geq |N_G(V'_1)|$. So, by simplifying $|V'_1|+\varepsilon_1+|Z_1| = |U'_1| + |Z_1| \geq |V'_1|+k$, we get $\varepsilon_1+|Z_1| \geq k$. In a similar way, we get that $\varepsilon_2+|Z_1| \geq k$. 

On one hand, since $|Z_2|+k=|Z_1|+ \varepsilon_1+\varepsilon_2\geq k+\varepsilon_2$, we obtain that $|Z_2|\geq \varepsilon_2$. On the other hand, there is $|Z_1|+|Z_2|<k\leq |Z_1|+\varepsilon_2$, so $|Z_2|<\varepsilon_2$, a contradiction.

Finally, the last part of the thesis of the theorem follows from the previous analyses.
\end{proof}

\begin{thm}\label{kappa}
Let $n,m$ be positive integers such that $1<m<n$. Let $k=n-m$. Then, for any $k$-critical-bipartite graph $G=(U,V;E)$ with $|U|=n$ and $|V|=m$, there is $\kappa(G) \geq \min\{\delta(G),k\}$.
\end{thm}

\begin{proof}
Let $G=(U,V;E)$ be a $k$-critical-bipartite graph of order $(n,m)$. Towards a contradiction, suppose that $Z$ is a {vertex cut} for two vertices $x$ and $y$ in $G$ with $|Z| < \min\{\delta_U(G),k\}$. Let $G'=G-Z$.

If $x,y\in V$, then $|Z|\geq k$ by Theorem~\ref{kappaV}, a contradiction. For $x,y\in U$, we have $|Z|\geq\min\{\delta_U(G),k\}$ by Theorem~\ref{kappaU}, a contradiction. Finally, consider the case where $x\in U$ and $y\in V$. Choose $x'\in U' \setminus \{x\} \cap N_G'(y)$. Such a vertex exists since $G$ is $k$-critical and $n \geq k+2 > |Z|+2$. So $Z$ is a vertex cut for $x$ and $x'$, and the thesis holds by Theorem~\ref{kappaU}.
\end{proof}

By applying Theorems \ref{kappaV}, \ref{kappaU}, and \ref{kappa} to Construction~\ref{cons2}, we get the following corollary that shows that the given lower bounds are tight.

\begin{cor}
Let $n,m$ be positive integers such that $1<m<n$. Let $k=n-m$ and $\overline{G}_{n,m}=(U,V; E)$ be a graph given by Construction~\ref{cons2}. Then the following properties hold: 
\begin{itemize}
\item $\kappa_V(\overline{G}_{n,m}) \in \{k,k+1\}$,
\item $\kappa_U(\overline{G}_{n,m})=\delta_U(\overline{G}_{n,m})$,
\item $\kappa(\overline{G}_{n,m})=\delta(\overline{G}_{n,m})$.
\end{itemize}
Moreover, $\kappa_V(\overline{G}_{n,m}) = k$ if $k=1$.
\end{cor}

\begin{proof}
Since $\Delta_V(\overline{G}_{n,m})=k+1$, by Theorem \ref{kappaV}, there is $\kappa_V(\overline{G}_{n,m}) \in \{k,k+1\}$. Since $\delta_U(\overline{G}_{n,m})=\floor{\frac{m(k+1)}{n}}$ and $\frac{m(k+1)}{n}<k+1$ for $n>m>1$, by Theorem \ref{kappaU}, there is $\kappa_U(\overline{G}_{n,m})=\delta_U(\overline{G}_{n,m})$. By Theorem \ref{kappa}, there is $\kappa(\overline{G}_{n,m})=\delta(\overline{G}_{n,m})$. Finally, the case where $k=1$ is easy to check (for example, consider removing $u_0$ in the graph $\overline{G}_{6,5}$ in Figure \ref{newcons}). 
\end{proof}

Note that given positive integer values $n,m$ such that $n>m>1$ and $k=n-m$, for any $\kappa\in\{1,\ldots,m\}$, there exists a $k$-critical-bipartite graph $G=(U,V; E)$ of order $(n,m)$ with connectivity $\kappa(G)=\kappa$. Indeed, if $\kappa=m$, then $G=K_{n,m}$, otherwise (i.e. $\kappa<m$)  let $G'=(V,U;E)$ be a complete bipartite graph $K_{n,m}$. If we pick any vertex $u\in U$  and delete $m-\kappa$ incident edges with $v$, the obtained graph $G=(U,V;E)$ is $k$-critical and $\kappa$-connected.

\section{Final remarks}\label{sec:final_remarks}

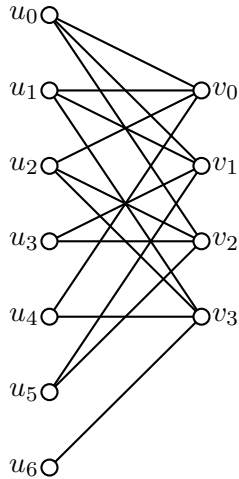
\begin{figure}[t]
\begin{center}
\begin{tikzpicture}[scale=1,style=thick,x=1cm,y=1cm]
\def\vr{3pt} 

\path (0,0) coordinate (u7);
\path (0,1) coordinate (u6);
\path (0,2) coordinate (u5);
\path (0,3) coordinate (u4);
\path (0,4) coordinate (u3);
\path (0,5) coordinate (u2);
\path (0,6) coordinate (u1);
\path (2,2) coordinate (v4);
\path (2,3) coordinate (v3);
\path (2,4) coordinate (v2);
\path (2,5) coordinate (v1);
\draw (u7) -- (v4) -- (u5) -- (v1) -- (u1) -- (v2) -- (u2) -- (v4);
\draw (u6) -- (v3) -- (u1);
\draw (u6) -- (v2) -- (u4) -- (v3);
\draw (v4) -- (u3) -- (v1) -- (u2);
\draw (u3) -- (v3);

\draw (u1) [fill=white] circle (\vr);
\draw (u2) [fill=white] circle (\vr);
\draw (u3) [fill=white] circle (\vr);
\draw (u4) [fill=white] circle (\vr);
\draw (u5) [fill=white] circle (\vr);
\draw (u6) [fill=white] circle (\vr);
\draw (u7) [fill=white] circle (\vr);

\draw (v1) [fill=white] circle (\vr);
\draw (v2) [fill=white] circle (\vr);
\draw (v3) [fill=white] circle (\vr);
\draw (v4) [fill=white] circle (\vr);

\draw[anchor = east] (u7) node {$u_6$};
\draw[anchor = east] (u6) node {$u_5$};
\draw[anchor = east] (u5) node {$u_4$};
\draw[anchor = east] (u4) node {$u_3$};
\draw[anchor = east] (u3) node {$u_2$};
\draw[anchor = east] (u2) node {$u_1$};
\draw[anchor = east] (u1) node {$u_0$};
\draw[anchor = west] (v4) node {$v_3$};
\draw[anchor = west] (v3) node {$v_2$};
\draw[anchor = west] (v2) node {$v_1$};
\draw[anchor = west] (v1) node {$v_0$};


\end{tikzpicture}
\caption{Minimum $k$-critical-bipartite graph with $\delta_U < \floor{\frac{(n-m+1)m}{n}}$}\label{small_delta}

\end{center}
\end{figure}

Let $G=(U,V;E)$ with $|U|=n$, $|V|=m$ $n>m>1$, $k=n-m$ be a minimum $k$-critical-bipartite graph. Then $\delta_V=k+1$, therefore $\kappa_V\in\{k,k+1\}$ by Theorem \ref{kappaV}. For example, note that $\kappa_{V}(\overline{G}_{6,5})=1$ (see Figure~\ref{newcons}). Therefore we pose the following open problem.

\begin{prob}\label{prob_kappaV}
Characterize all minimum $k$-critical-bipartite graphs for which $\kappa_V=k$.
\end{prob}

Recall that for any minimum $k$-critical-bipartite graph $G=(U,V;E)$ of order $(n,m)$, with $k=n-m+1)$, there is $|E|=(k+1)m$ and $\Delta_U=\ceil{\frac{(k+1)m}{n}}$. And, to have  $(k+1)m$ edges, the number of vertices in $U$ of degree $\Delta_U$ has to be at least $(n-m+1)m-n\floor{\frac{(k+1)m}{n}}$. But there is some flexibility with respect to the degree of other vertices in $U$: there may be $\delta_U < \floor{\frac{(k+1)m}{n}}$ (see Figure~\ref{small_delta} for an example). So we pose the following open problem.

\begin{prob}\label{prob_small_delta}
Determine if there exist minimum $k$-critical-bipartite graphs $G=(U,V;E)$ of order $(n,m)$ with $\delta_U=\delta$ for any $\delta\in\{1,2.\ldots,\floor{\frac{m(k+1)}{n}}-1\}$.
\end{prob}

The property of being $k$-critical-bipartite, and fault-tolerance in general, has strong relations with connectivity (besides the results presented in this paper, see, for example, the work of Cichacz et al.~\cite{kliki}). Let us recall that, by Theorem~\ref{th:tilde}, $G$ is $k$-critical-bipartite if and only if $\tilde{G}$ is $k$-extendable, and there is the following result by Robertson et al.
\begin{thm}[\cite{RST99}]\label{thm_RST99}
    Let $G=(U,V; E)$ be a connected bipartite graph, let $M$ be a perfect matching in $G$, and let $k \geq 1$ be an integer. Then $G$ is $k$-extendable if and only if $D(G,M)$ is strongly $k$-connected, where $D(G,M)$ is the directed graph obtained by directing every edge from $U$ to $V$, and contracting every edge of $M$.
\end{thm}
Since the respective auxiliary graphs can be constructed in polynomial time, testing if a graph is $k$-critical-bipartite reduces to testing connectivity in directed graphs, which can also be done efficiently (see, for example, the results of Henziger et al. \cite{H00}). But, given a bipartite graph $G=(U,V;E)$, besides testing if $G$ is $k$-critical-bipartite, it is valuable for applications to find a minimum supergraph (adding edges, augmentation) or minimum subgraph (removing edges, sparsification) of $G$ that is $k$-critical-bipartite. However, unlike testing connectivity, the corresponding edge modification problems tend to be harder and not well understood (see the work of Crespelle et al.~\cite{Crespelle2023} for a recent review). We believe that the relations between $k$-critical-bipartiteness and connectivity will permit us to adapt methods developed for edge modification problems related to connectivity to work with $k$-critical-bipartiteness. And characterizing minimum $k$-critical-bipartite graphs is a valuable step in this direction. We terminate with the following open problem.

\begin{prob}\label{prob_complexity}
Given a bipartite graph $G=(U,V;E)$ of order $(n,m)$, what is the complexity of finding a minimum supergraph (subgraph) of $G$ that is $k$-critical-bipartite.
\end{prob}

\bibliographystyle{abbrvnat}
\bibliography{matchings}

\end{document}